\documentclass[12pt]{amsart}
\usepackage[colorlinks=true,pagebackref,hyperindex,citecolor=NavyBlue,linkcolor=Mahogany]{hyperref}
\usepackage[dvipsnames]{xcolor}
\usepackage{verbatim,dsfont}
\usepackage{amsmath}
\usepackage{amsfonts}
\usepackage{amssymb}
\usepackage{color}
\usepackage[all]{xy}  
\usepackage{enumerate}
\usepackage[top=1in, bottom=1in, left=0.9in, right=0.9in]{geometry}
\usepackage{mathrsfs}

\usepackage{stmaryrd}
\usepackage{verbatim}
\usepackage{tikz}
\usetikzlibrary{shapes.geometric}
\usetikzlibrary{calc,shadows}
\usetikzlibrary{decorations.markings}

\renewcommand{\c}{\mathsf{c}}

\theoremstyle{definition}
\newtheorem{theorem}{Theorem}[section]
\newtheorem{theoremx}{Theorem}
 
\newtheorem{corollaryx}{Corollary}
\numberwithin{equation}{section}

\newtheorem{corollary}[theorem]{Corollary}
\newtheorem{lemma}[theorem]{Lemma}

\newtheorem{notation}[theorem]{Notation}

\newtheorem*{claim*}{Claim}

\theoremstyle{definition}
\newtheorem{definition}[theorem]{Definition}

\newtheorem{example}[theorem]{Example}

\newtheorem{conjecture}[theorem]{Conjecture}
\newtheorem{remark}[theorem]{Remark}
\newcommand{\e}{\operatorname{e}}

\newtheoremstyle{TheoremNum}
        {8pt}{8pt}              %%% space between body and theorem
        {\upshape}                      %%% theorem body font
        {}                              %%% Indent amount (empty = no indent)
        {\bfseries}                     %%% theorem head font
        {.}                             %%% Punctuation after theorem head
        {.5em}                             %%% Space after theorem head
        {\thmname{#1}\thmnote{ \bfseries #3}}%%% Thm head spec
  \theoremstyle{TheoremNum}

\newcommand{\m}{\mathfrak{m}}

\newcommand{\NN}{\mathbb{N}}
\newcommand{\ZZ}{\mathbb{Z}}

\newcommand{\g}{\mathfrak{g}}

\newcommand{\Depth}{\operatorname{depth}}

\newcommand{\Ass}{\operatorname{Ass}}
\newcommand{\LL}{\mathds L}

\newcommand{\Char}{\operatorname{char}}

\newcommand{\Ht}{\operatorname{ht}}

\newcommand{\HF}{\operatorname{HF}}

\renewcommand{\b}{\mathfrak{b}}

\renewcommand{\d}{\mathbf{d}}

\newcommand{\ov}[1]{\overline{#1}}

\renewcommand{\a}{\mathfrak{a}}

\newcommand{\PP}{\mathbb{P}}
\renewcommand{\leq}{\leqslant}
\renewcommand{\geq}{\geqslant}

\newcommand{\Assh}{\operatorname{Assh}}
\newcommand{\f}{\mathfrak{f}}
\newcommand{\kk}{\Bbbk}

\newcommand{\soc}{\operatorname{soc}}

\newcommand{\rev}[1]{\operatorname{{\rm rev}_{#1}}}

\newcommand{\ale}[1]{{\color{red} \sf $\star$ Alessandro: [#1]}}

\newcommand{\EGH}[2]{{\rm EGH}_{{#1},{#2}}}
\newcommand{\x}{\mathsf{x}}

\title[A Cayley-Bacharach theorem for points in $\PP^n$]{A Cayley-Bacharach theorem for points in $\PP^n$}
\author{Giulio Caviglia}
\address{Department of Mathematics, Purdue University, 150 N. University Street, West Lafayette, IN 47907-2067, USA}
\email{gcavigli@purdue.edu}
\author{Alessandro De Stefani}
\address{Dipartimento di Matematica, Universit{\`a} di Genova, Via Dodecaneso 35, 16146 Genova, Italy}
\email{destefani@dima.unige.it}


\begin{document}

\begin{abstract}
We prove a Cayley-Bacharach-type theorem for points in projective space $\PP^n$ that lie on a complete intersection of $n$ hypersurfaces. This is made possible by new bounds on the growth of the Hilbert function of almost complete intersections. 

%For $n\leq 3$ our theorems provide sharp estimates. In $\PP^3$, for example, we prove that every surface of degree $d$ containing at least $d^3-d^2+d+1$ points of a complete intersection of three surfaces of degree $d$ must contains all $d^3$ of them. 
\end{abstract} 

\maketitle

\section{Introduction}

Let $\kk$ be an algebraically closed field. If $C_1$ and $C_2$ are two cubics in $\PP^2_\kk$ which meet in $9$ points, and $X$ is a cubic passing through $8$ points of $C_1 \cap C_2$, then $X$ contains the nineth point of $C_1 \cap C_2$ as well. This well-known statement extends Pappus's and Pascal's theorems, and it is one version of a series of results which are referred to as Cayley-Bacharach theorems. We refer the interested reader to the seminal work of Eisenbud, Green and Harris \cite{EGH_CB}, and to recent work of Kreuzer, Long and Robbiano \cite{KLR} for a detailed and fascinating history on the subject.

\begin{figure}[h]
\centering
\includegraphics[width=0.5\textwidth]{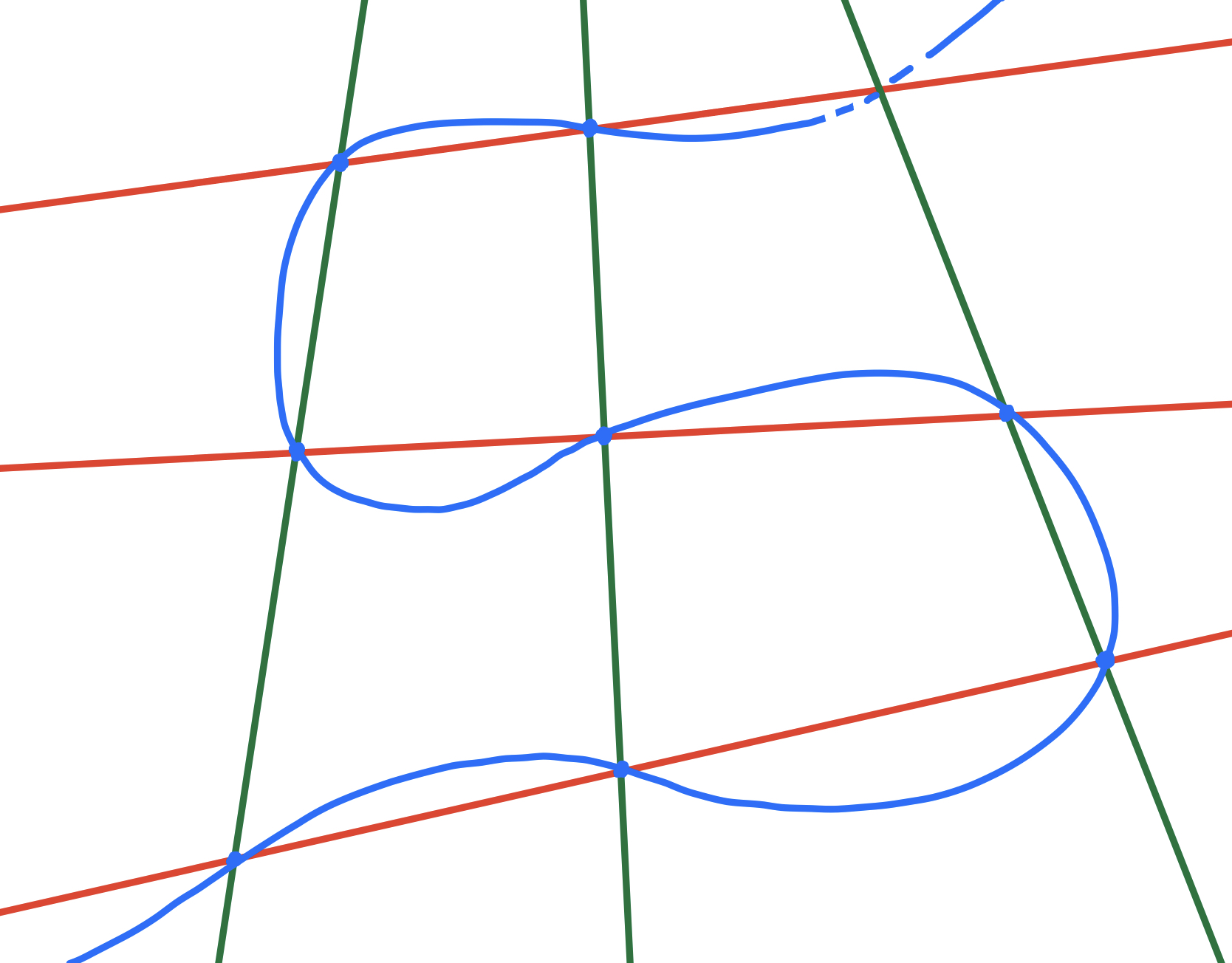}
\caption{A sketch of the case in which $C_1$ and $C_2$ are a union of three lines.}
\end{figure}

More generally, if $C_1,C_2 \subseteq \PP^2$ are two curves of degrees $d_1 \leq d_2$ meeting transversely in $d_1d_2$ points, the Cayley-Bacharach theorem states that, if a curve $X$ of degree $D=d_1+d_2-3$ passes through all but possibly one point of $C_1 \cap C_2$, then it must contain all $d_1d_2$ points of $C_1 \cap C_2$. In the literature, there have been several efforts to extend this theorem to a more general setup \cite{GH, Tan, Li, EL,KLR}. However, in most cases, the obtained results still require the hypersurface to pass through at least all but one point. In \cite{EGH_CB}, Eisenbud, Green and Harris suggest a different direction in which this theorem can be pushed. Namely, one can require $X$ to contain all but a given number of points of $C_1 \cap C_2$, balancing off this additional freedom by putting more restrictions on the degree of $X$. This leads to a new series of conjectured inequalities on multiplicities of almost complete intersections (see \cite[Conjecture CB12]{EGH_CB}). More specifically, \cite[Conjecture CB12]{EGH_CB} can be restated by saying that the multiplicity of an almost complete intersections is bounded above by the multiplicity of a special monomial almost complete intersection of the same degrees, which in Section \ref{Section Gorenstein} we denote $\LL(\d;D)$. By its nature, this upper bound is sharp, if true. In the literature, this improved version has been called the General Cayley-Bacharach conjecture (see \cite{GK}). However, in this article we will refer to the above as Cayley-Bacharach, since it is the only one we will focus on.

%The main goal of this article is to obtain Cayley-Bacharach-type theorems for $\PP^n$ by collecting and improving known results on the EGH conjecture for almost complete intersections.

In $\PP^2$, \cite[Conjecture CB12]{EGH_CB} and hence the Cayley-Bacharach theorem follow from a stronger conjecture of Eisenbud, Green and Harris (henceforth EGH), see Conjecture \ref{Conjecture EGH}, which is known to be true in this case by \cite{Richert,Cooper}. %In order to prove a sharp Cayley-Bacharach theorem, the full EGH conjecture is not needed. In fact, it is sufficient to prove that the multiplicity of almost complete intersections. 
In $\PP^3$, \cite[Conjecture CB12]{EGH_CB} and the Cayley-Bacharach theorem have been proved by Geramita and Kreuzer \cite[Corollary 4.4]{GK}. 

In Section \ref{Section Gorenstein}, we refine the Cayley-Bacharach inequality on multiplicities of almost complete intersections of height three, and we obtain the following upper bound on their Hilbert functions.  % The key result in this direction is Theorem \ref{Theorem Artinian}, which allows us to reduce to the Artinian case. In fact, in characteristic $p>0$ it is not true in general that an almost complete intersection remains such after reducing modulo a general linear form (see Example \ref{Example reduction}). 
%Let $\HF(-)$ and $\e(-)$ denote, respectively, the Hilbert function and the multiplicity of a graded module. The main result of Section \ref{Section Gorenstein} is the following.
\begin{theoremx}[see Theorem \ref{Theorem EGH ACI}]  \label{THMX B} Let $S = \kk[x_1,\ldots,x_n]$, and $\f=(f_1,f_2,f_3)$ be a complete intersection of degrees $\d=(d_1,d_2,d_3)$. Let $G$ be an element of degree $D \leq d_1+d_2+d_3-3$ such that $G \notin \f$, and $\a=\f+(G)$ has height three. Then $\HF(S/\a) \leq \HF(S/\LL(\d;D))$, where $\LL(\d;D) = (x_1^{d_1},x_2^{d_2},x_3^{d_3},U_D)$ and $U_D$ is the largest monomial with respect to the lexicographic order which has degree $D$ and does not belong to $(x_1^{d_1},x_2^{d_2},x_3^{d_3})$. %In particular, $\e(S/\a) \leq \e(S/\LL)$.
\end{theoremx}

Theorem \ref{THMX B} in particular gives that the multiplicity (denoted $\e(-)$) of an almost complete intersection of degrees $(d_1,d_2,d_3;D)$ is at most $\e(S/\LL(\d;D))$, as conjectured in \cite[Conjecture CB12]{EGH_CB}. We would like to point out that the statement on Hilbert functions is stronger than the corresponding one on multiplicities. In fact, the standard techniques which usually allow to reduce to the Artinian case might fail for this purpose (see Example \ref{Example reduction}). We also note that the stronger statement on Hilbert function rather than just on multiplicities is needed in Section \ref{Section ACI} to improve the Cayley-Bacharach theorem in $\PP^n$, Theorem \ref{THMX Pn}, in a special case (see Theorem \ref{Theorem delta2}).

\begin{comment} Given integers $1 \leq d_1 \leq d_2 \leq d_3$, and $1 \leq D \leq d_1+d_2+d_3-3$, in Section \ref{Section Gorenstein} we construct a new sequence $\c = (c_1,c_2,c_3)$. For example, if $d_1=d_2=d_3=D=3$, then $\c=(1,2,3)$. The following is the analogue of Theorem \ref{THMX P2} for $\PP^3$. The bound we obtain is the one predicted by the EGH conjecture, and it is therefore sharp. 

\begin{theoremx}[Cayley-Bacharach in $\PP^3$] \label{THMX P3} Let $\Gamma \subseteq \PP^3$ be a complete intersection of three surfaces of degrees $(d_1,d_2,d_3)$. If $X$ is a surface of degree $D \leq \sigma=d_1+d_2+d_3-3$ which contains at least $d_1d_2d_3-c_1c_2c_3+1$ points of $\Gamma$, then $X$ contains $\Gamma$.
\end{theoremx}

For example, if $\Gamma$ is a complete intersection of three cubic surfaces in $\PP^3$, and $X$ is another cubic surface containing at least $22$ points of $\Gamma$, then $X$ must in fact contain all $27$ points of $\Gamma$. 
\end{comment}

Using Theorem \ref{THMX B}, we immediately recover the Cayley-Bacharach theorem for points in $\PP^3$.
\begin{corollaryx}[Cayley-Bacharach in $\PP^3$] \label{Corollary P3} Let $\Gamma \subseteq \PP^3$ be a complete intersection of three surfaces of degrees $\d=(d_1,d_2,d_3)$. If $X$ is a surface of degree $D \leq \sigma=d_1+d_2+d_3-3$ which contains at least $d_1d_2d_3-\e(S/\LL(\d;D))+1$ points of $\Gamma$, then $X$ contains $\Gamma$.
\end{corollaryx}

We refer to Section \ref{Section Gorenstein} for an explicit way to compute $\e(S/\LL(\d;D))$ in terms of a new sequence $\c=(c_1,c_2,c_3)$, constructed from $\d$ and $D$. To give an example, if a surface of degree $D$ in $\PP^3$ contains at least $D^3-D^2+D+1$ points of a complete intersection of three surfaces of degree $D$, then it must contains all $D^3$ of them.

%The proof of Corollary \ref{Corollary P3} goes as follows: let $\f\subseteq S=\kk[\PP^3]$ be a complete intersection defining $\Gamma$, and $G$ be a form of degree $D$ defining $X$. If $G \notin \f$, by Theorem \ref{THMX B} the multiplicity of $S/(\f+(G))$ is at most $d_1d_2d_3-\e(S/\LL(\d;D))$. So $X$ contains at most this number of points of $\Gamma$, which contradicts our assumptions.

%Theorem \ref{THMX B} can be viewed as a weak version of the EGH conjecture. Such a
An analogue of Theorem \ref{THMX B} is not known, in general, for almost complete intersections of codimension higher than three. However, a result of Francisco \cite{F} gives an upper bound on the Hilbert function of any almost complete intersection in one specific degree. 
In Section \ref{Section ACI}, we exhibit upper bounds for the multiplicity of almost complete intersections of any height combining a repeated use of Francisco's theorem with several other techniques (see Theorems \ref{THM multiplicity} and \ref{Theorem symmetric}). While our estimates are not in general as sharp as the ones predicted by \cite[Conjecture CB12]{EGH_CB}, they significantly improve the best known upper bounds, due to Engheta \cite{E} and later extended by Huneke, Mantero, McCullough and Seceleaenu \cite{HMMCS}, in all those circumstances in which the latter are not already sharp (see Remark \ref{Remark HMMCS}). 

Using these estimates, we obtain a Cayley-Bacharach-type theorem in $\PP^n$. We refer the reader to Section \ref{Section ACI} and Theorem \ref{THM Pn} for the definition of the integer $\delta(\d;D)$ which appears in the statement of the theorem.

\begin{theoremx}[Cayley-Bacharach in $\PP^n$] \label{THMX Pn} Let $\Gamma \subseteq \PP^n$ be a complete intersection of degrees $\d = (d_1,\ldots,d_n)$. If $X$ is a hypersurface of degree $D < \sigma = \sum_{i=1}^n(d_i-1)$ which contains at least $\delta(\d;D) = \prod_{i=1}^n d_i - \sum_{m=D+1}^{\tau_-} \varphi_m - \sum_{m=D+1}^{\tau^+} \varphi_m - 1$ points of $\Gamma$, then $X$ contains $\Gamma$. 
\end{theoremx}

As an explicit consequence of Theorem \ref{THMX Pn}, if a cubic hypersurface in $\PP^{2n}$ contains at least $3^{2n} - (6n^2-8n+3)$ points of a complete intersection of $2n$ cubics, then it contains all of them. 

As another application, if a hypersurface of degree $D \leq n$ in $\PP^n$ contains at least $2^n-\lfloor \frac{3(n-D)^2+1}{4} \rfloor$ points of a complete intersection of $n$ quadrics, then it contains all of them.

Finally, combining Theorem \ref{THMX B} and our new bounds on the multiplicity of almost complete intersection of any height, we improve Theorem \ref{THMX Pn} in case the degree $D$ of the hypersurface $X$ is less than $d_4$, see Theorem \ref{Theorem delta2}. As already pointed out, for this result it is crucial that Theorem \ref{THMX B} gives an upper bound on the Hilbert function of an almost complete intersection of codimension three, rather than on its multiplicity alone. In this scenario, when Theorem \ref{Theorem delta2} can be applied it drastically improves Theorem \ref{THMX Pn}, and it often allows to obtain estimates which are rather close to the optimal ones of \cite[Conjecture CB12]{EGH_CB} (see Example \ref{Example (4,4,4,10;4)}).

\subsection*{Acknowledgments} We thank Martin Kreuzer and Lorenzo Robbiano for pointing out some inaccuracies on a previous version of the paper.

\section{Almost complete intersections of codimension three} \label{Section Gorenstein}

The goal of this section is to prove an upper bound on the Hilbert function of almost complete intersections of codimension three. 

We start by setting up some notation, which will be used throughout the article. In what follows, $S= \kk[x_1,\ldots,x_n]$ denotes a polynomial ring over any field $\kk$. We consider the standard grading on $S$, that is, $\deg(x_i)=1$ for all $i$. We denote by $\m$ the irrelevant maximal ideal $(x_1,\ldots,x_n)S$ of $S$. We adopt the convention that a sum $\sum_i^j (-)$ equals zero whenever $j<i$. Given a finitely generated graded $S$-module $M$ we write $\HF(M)$ for the Hilbert function of $M$, that is, the numerical function $j \in \ZZ \mapsto \HF(M;j) = \dim_\kk(M_j)$, and $\e(M)$ for its multiplicity. Given two graded $S$-modules $M$ and $N$, we write $\HF(M) \leq \HF(N)$ to mean $\HF(M;j) \leq \HF(N;j)$ for all $j \in \ZZ$.

Let $\d = (d_1,\ldots,d_h) \in \NN^{h}$, with $1 \leq d_1 \leq \cdots \leq d_h$. We denote by $(\x^\d)$ the ideal $(x_1^{d_1},\ldots,x_h^{d_h})$ of $S$. An ideal $\LL \subseteq S$ is a called a $\d$-LPP ideal if we can write $\LL=(\x^\d) + L$, where $L$ is a lexicographic ideal (see \cite[Definitions 4 and 5]{CM}). We now state the current version of the Eisenbud-Green-Harris conjecture \cite{EGH} (see \cite{CM}). 

\begin{conjecture}[$\EGH{\d}{n}$] \label{Conjecture EGH}
Let $I$ be a homogeneous ideal of $S=\kk[x_1,\ldots,x_n]$, containing a regular sequence of degrees $\d = (d_1,\ldots,d_h)$. There exists a $\d$-LPP ideal containing $(\x^\d)$ which has the same Hilbert function as $I$. 
\end{conjecture} 

It is easy to show that $I$ satisfies $\EGH{\d}{n}$ if and only if the following holds: for any $j \in \ZZ$, if $L$ is a lexicographic ideal generated in degree $j$ such that the $\d$-LPP ideal $\LL(j)=(\x^{\d}) + L$ satisfies $\HF(S/I;j) = \HF(S/\LL(j);j)$, then $\HF(S/I;j+1) \leq \HF(S/\LL(j);j+1)$. Because of standard properties of lexsegment ideals, the latter is also equivalent to $\LL(j)_{\geq j+1} \subseteq \LL(j+1)$, where $\LL(j)_{\geq j+1}$ denotes the ideal generated by the elements of $\LL(j)$ of degree at least $j+1$.

Conjecture $\EGH{\d}{n}$ is known, among other cases, if $\d=(d_1,d_2)$ \cite{Richert, Cooper}, if $I$ contains a monomial regular sequence of degrees $\d$ \cite{CL,MP,CK1}, if the degrees of the forms in the regular sequence grow sufficiently fast \cite{CM}, if $I=Q_1+Q_2$, where $Q_1$ is generated by a regular sequence of quadrics and $Q_2$ is generated by general quadratic forms \cite{HP,G}, if the regular sequence factors as a product of linear forms \cite{A}, and if $d_1 = \ldots = d_h=2$ and $h \leq 5$ \cite{GuHo}. In general, however, the conjecture is wide-open. 

One more case in which the conjecture is known is for the class of minimally licci ideal, defined by Chong \cite{Chong}. Chong proves that, if $\g \subseteq S$ is an ideal of height three which contains a regular sequence of degrees $\d$ among its minimal generators, and such that $S/\g$ is Gorenstein, then $\g$ satisfies $\EGH{\d}{n}$. The condition that the regular sequence is part of a minimal generating set for $\g$ can actually be removed, as the following lemma shows.

\begin{lemma} \label{Lemma Gorenstein minimal} Let $\g \subseteq \kk[x_1,\ldots,x_n]$ be an ideal of height three, containing a regular sequence of degrees $\d = (d_1,d_2,d_3)$. If $S/\g$ is Gorenstein, then $\g$ satisfies $\EGH{\d}{n}$. 
\end{lemma}
\begin{proof}
We may harmlessly assume that $\kk$ is infinite. Since $\g$ contains a regular sequence of degrees $\d$, we can find a regular sequence $f_1',f_2',f_3'$ of degrees $\d'$, with $d_i' \leq d_i$ for $i=1,2,3$, among the minimal generators of $\g$. By \cite[Corollary 11]{Chong}, there exists a $\d'$-LPP-ideal $\LL$ which has the same Hilbert function as $\g$. Since $\LL$ contains $(\x^{\d'})$ which, in turn, contains $(\x^\d)$, by \cite[Theorem 1.2]{MP} we can find a $\d$-LPP ideal with the same Hilbert function as $\LL$, and this concludes the proof. 
\end{proof}

We now turn our attention to almost complete intersections. 

\begin{definition} \label{Defn ACI}
Let $\a$ be a homogeneous ideal of $S$. We say that $\a$ is an almost complete intersection of degrees $(\d;D) = (d_1,\ldots,d_h;D)$ if $\Ht(\a) = h$, and we can write $\a=\f+(G)$, where the ideal $\f=(f_1,\ldots,f_h)$ is generated by a regular sequence of degrees $d_1 \leq \cdots \leq d_h$, and $G$ is an element of degree $D$ which does not belong to $\f$. 
\end{definition}

Observe that we do not require that $\a$ is minimally generated by $h+1$ elements. For example, according to our definition, the ideal $\a=(x_1^2,x_2^3)+(x_2^2)$ is an almost complete intersection of degrees $(2,3;2)$, but also a complete intersection of degrees $(2,2)$. What is important to observe, though, is that an almost complete intersection of degrees $(\d;D)$ cannot be generated by a regular sequence of degrees $\d$. %This choice is the most convenient for us in order to restate our Cayley-Bacharach-type theorems in algebraic terms.

\begin{notation}
Given integers $(\d;D) = (d_1,\ldots,d_h;D)$, with $D \leq \sum_{i=1}^h(d_i-1)$, we let $\LL(\d;D) = (\x^\d) + (U_D)$ be the $\d$-LPP ideal of $S = \kk[x_1,\ldots,x_n]$ which is an almost complete intersection of degrees $(\d;D)$. In other words, $U_D$ is the largest monomial with respect to the lexicographic order which has degree $D$, and does not belong to $(\x^\d)$. 
\end{notation}

In order to apply Lemma \ref{Lemma Gorenstein minimal} to obtain an upper bound on the multiplicity of almost complete intersections, we will use partial initial ideals with respect to the weight order $\omega=(1,1,\ldots,1,0)$. For unexplained notation and terminology, we refer to \cite{CK} and \cite{CDS}, where such weight order is denoted by $\rev{1}$. 
For convenience of the reader, we recall the main features of such an object. Let $I$ be a homogeneous ideal in $S = \kk[x_1,\ldots,x_n]$, and assume that $\kk$ is infinite. After performing a sufficiently general change of coordinates, there is a vector space decomposition ${\rm in}_\omega(I) = \bigoplus_{j \geq 0} I_{[j]}x_n^j$, where each $I_{[j]}$ is an ideal in $\ov{S} = \kk[x_1,\ldots,x_{n-1}]$. This decomposition is analogous to the one in \cite[Section 6]{GreenGin}, where Green constructs partial elimination ideals for the lexicographic order. Observe that $I_{[0]}$ is the ideal defining the hyperplane section $S/(I+(x_n))$ viewed inside $\ov{S} \cong S/(x_n)$. In characteristic zero, the ideals $I_{[j]}$ automatically satisfy $\ov{\m} I_{[j+1]} \subseteq I_{[j]}$ for all $j \geq 0$, where $\ov{\m} = (x_1,\ldots,x_{n-1})\ov{S}$ (see \cite[Theorem 3.2]{CS-JCA}). We will refer to this phenomenon as stability of partial general initial ideals. We may achieve this also in characteristic $p>0$, without altering the relevant features of ${\rm in}_{\omega}(I)$, by recursively applying general distractions and partial initial ideals with respect to $\omega$. For a description of this process, see the proof of \cite[Theorem 4.1]{CK}, or \cite[Proposition 1.4]{CS-MA}. We point out that, while this process may change the ideals $I_{[j]}$, it can only enlarge $I_{[0]}$. 

We record these facts in a lemma, for future use.
\begin{lemma} \label{Lemma decomposition} Let $I$ be a homogeneous ideal in $S=\kk[x_1,\ldots,x_n]$, where $\kk$ is an infinite field. With the notation introduced above, after performing a sufficiently general change of coordinates, there exist ideals $I_{[j]} \subseteq \ov{S}$ and an ideal $\widetilde{I} = \bigoplus_{j\geq 0} I_{[j]} x_n^j$ of $S$ such that
\begin{itemize}
\item $\ov{\m} I_{[j+1]} \subseteq I_{[j]}$ for all $j \geq 0$.
\item $I+(x_n) \subseteq \widetilde{I}+(x_n)$, with equality if $\Char(\kk)=0$. 
\item $\HF(I) = \HF(\widetilde{I})$. 
\end{itemize}
\end{lemma}

Let $\a=\f+(G)$ be an almost complete intersection of degrees $(\d;D)$, with $D \leq \sum_{i=1}^h(d_i-1)$. In order to estimate the Hilbert function and the multiplicity of $S/\a$, it would be desirable to reduce to the Artinian case, without losing the relevant features of $\a$. In particular, if $y \in S$ is a linear form which is regular modulo $\f$, then it would be good to have that the image of $G$ is non-zero in $S/(\f+(y))$, at least for a general choice of $y$. While this is true if $\kk$ has characteristic zero as a consequence of the proof of the forthcoming Theorem \ref{Theorem Artinian}, it may be false in prime characteristic.
\begin{example} \label{Example reduction} Let $S=\kk[x_1,x_2,x_3]$, with $\Char(\kk)=p>0$, and let $\a=(x_1^{p^2},x_2^{p^2}) + (x_1x_3^{p^2})$, which is an almost complete intersection of degrees $(p^2,p^2;p^2+1)$. A linear regular element for $S/(x_1^{p^2},x_2^{p^2})$ is necessarily of the form $y = \lambda_1x_1+\lambda_2x_2+\lambda_3x_3$, with $\lambda_i \in \kk$, and $\lambda_3\ne 0$. It follows that  $G=x_1x_3^{p^2}$ is zero in $S/(x_1^{p^2},x_2^{p^2},y)$ for any choice of $y$ as above. 
\end{example}

The next theorem allows us to tackle the issue illustrated by the previous example. Even if the image of $G$ can be zero in $S/(\f+(y))$ for any general linear form $y$, using the techniques described above we can still reduce to the Artinian case in order to estimate the Hilbert function of an almost complete intersection, even in characteristic $p>0$.

\begin{theorem} \label{Theorem Artinian}
Let $S=\kk[x_1,\ldots,x_n]$, where $\kk$ is an infinite field, and $\a = \f+(G)$ be an almost complete intersection of degrees $(\d;D)=(d_1,\ldots,d_h; D)$. If $D \leq \sigma=\sum_{i=1}^h (d_i-1)$, then there exists an Artinian almost complete intersection 
$\ov\a \subseteq \ov{S}=\kk[x_1,\ldots,x_h]$ of degrees $(\d;D)$ such that $\HF(S/\a) \leq \HF(S/\ov{\a}S)$. In particular, $\e(S/\a) \leq \e(\ov{S}/\ov{\a})$. 
\end{theorem}
\begin{proof} 
It suffices to show the inequality on Hilbert functions and, to prove it, we proceed by induction on $\dim(S/\a) \geq 0$. The base case is trivial, so assume $\dim(S/\a)>0$. After a general change of coordinates we may find a decomposition $\widetilde{\a} = \bigoplus_{j \geq 0} \a_{[j]}x_n^j$ for $\a$ as in Lemma \ref{Lemma decomposition}, and we may further assume that $x_n$ is regular for $\f$.

Let $S'=\kk[x_1,\ldots,x_{n-1}]$, and $\m' = (x_1,\ldots,x_{n-1})S'$. Since $x_n$ is regular for $\f$, the elements ${\rm in}_\omega(f_1),\ldots,{\rm in}_\omega(f_h)$ form a regular sequence of degree $\d$, sitting necessarily inside $\a_{[0]}$. Let $\f'$ be the ideal they generate inside $S'$. 

Define $j= \inf\{t \geq 0 \mid \HF(\a_{[t]}x_n^t;D) \ne \HF(\f'x_n^t;D)\}$. Observe that $j$ is finite, since otherwise the condition $\HF(\a_{[t]}x_n^t;D) = \HF(\f'x_n^t;D)$ for all $t\geq 0$ would imply that $\HF(\a;D) = \HF(\widetilde{\a};D) = \HF( \bigoplus_{t \geq 0} \f' x_n^t ;D) = \HF(\f;D)$, contradicting our assumption that $G \in \a \smallsetminus \f$.

We claim that $j=0$. If not, let $H x_n^j \in \a_{[j]}x_n^j$ be an element of degree $D$, so that $H \in \a_{[j]} \smallsetminus \f'$ is an element of degree $D-j<D$. Observe that $H \in \f':\m'$ by stability, and because $\a_{[j-1]}$ coincides with $\f'$ up to degree $D-j+1$. It follows that $H$ represents a non-zero element of $\soc(S'/\f')$. If $\dim(S/\a)=1$, then we reach a contradiction since $\deg(H) < \sum_{i=0}^h (d_i-1)$, and the latter is the degree in which the socle of $S'/\f'$ is concentrated. If $\dim(S/\a)>1$, then $\dim(S'/\f') = \Depth(S'/
\f')>0$, therefore $\soc(S'/\f')=0$. A contradiction again.

Therefore $j=0$, and $\a' = \f'+(H)S'$ is an almost complete intersection of degrees $(\d;D)$, with $\dim(S'/\a') = \dim(S/\a)-1$. By induction, there exists an Artinian almost complete intersection $\ov{\a} \subseteq \ov{S}$ such that $\HF(S'/\a') \leq \HF(S'/\ov{\a}S')$. Because $\HF(S/\a) \leq \HF(S/\a'S)$, it follows that $\HF(S/\a) \leq \HF(S/\ov{\a}S)$, and the proof is complete. 
\end{proof}

Building from Chong's work in \cite{Chong}, and using Lemma \ref{Lemma Gorenstein minimal} and Theorem \ref{Theorem Artinian}, we can now prove the main result of this section.

\begin{theorem} \label{Theorem EGH ACI} Let $\a \subseteq S = \kk[x_1,\ldots,x_n]$ be an almost complete intersections of degrees $(\d;D)=(d_1,d_2,d_3;D)$, with $D \leq \sigma = d_1+d_2+d_3-3$. Then $\HF(S/\a) \leq \HF(S/\LL(\d;D))$.
\end{theorem}
\begin{proof}
We may assume that $\kk$ is infinite and, by Theorem \ref{Theorem Artinian}, that $S/\a$ is Artinian. Write $\a=\f+(G)$, where $\f$ is generated by a regular sequence of degrees $\d$, and $G$ is a form of degree $D$. By a standard argument of linkage (for instance, see \cite[Corollary 5.2.19]{Migliore}), for all $j \in \ZZ$ we get $\HF(S/\a;j) = \HF(S/\f;j) - \HF(S/\g;\sigma-j)$, where $\g=\f:\a$. Since $\g$ contains $\f$, and it defines a Gorenstein ring, by Lemma \ref{Lemma Gorenstein minimal} there is a $\d$-LPP  ideal $\LL$ with the same Hilbert function as $\g$. If we set $\b = (\x^\d): \LL$, then using linkage again we obtain that
\begin{align*}
\HF(S/\b;j) & = \HF(S/(\x^\d);j) - \HF(S/\LL;\sigma-j) \\
& = \HF(S/\f;j) - \HF(S/\g;\sigma-j) = \HF(S/\a;j)
\end{align*}
for all $j \in \ZZ$. By \cite[Theorem 1.2]{MP}, the monomial ideal $\b$ satisfies $\EGH{\d}{n}$. Therefore, there exists a $\d$-LPP ideal $\LL'$ with the same Hilbert function as $\b$. In particular, since $\LL'$ must contain $\LL(\d;D)$, we have that $\HF(S/\a) = \HF(S/\b) = \HF(S/\LL') \leq \HF(S/\LL(\d;D))$.
\end{proof}

\begin{remark} \label{Remark over socle} 
If $\a \subseteq S=\kk[x_1,\ldots,x_n]$ is an almost complete intersection of degrees $(\d;D)=(d_1,\ldots,d_h;D)$, with $D>\sigma = \sum_{i=1}^h(d_i-1)$, then the conclusion of Theorem \ref{Theorem EGH ACI} still holds, even without assuming that $h=3$. In fact, in this scenario we have that $\LL(\d;D) = (\x^\d) + (U_D)$, where $U_D=x_1^{d_1-1}\cdots x_h^{d_h-1}x_{h+1}^{D-\sigma}$. Iterating the argument used in the proof of Theorem \ref{Theorem Artinian}, we can find an almost complete intersection $\a' = \f'+(G') \subseteq S'=\kk[x_1,\ldots,x_{h+1}]$ of degrees $(\d;D)$ such that $\HF(S/\a) \leq \HF(S/\a'S)$. Moreover, we may assume that $x_{h+1}$ is regular modulo $\f'$. Since $\HF(\LL(\d;D)/(\x^\d);m) \leq 1$ for all $m \in \ZZ$, with equality if and only if $m \geq D$, it follows that $\HF(\a'/\f') \geq \HF(\LL(\d;D)/(\x^\d))$, because otherwise we would have $\a' \subseteq (\f')^{\rm sat} = \f'$. As a consequence, $\HF(S/\a) \leq \HF(S/\a'S) \leq \HF(S/\LL(\d;D))$.
\end{remark}

We now show how Theorem \ref{Theorem EGH ACI} allows to recover the Cayley-Bacharach theorem for points in $\PP^3$, which has been proved by Geramita and Kreuzer \cite[Corollary 4.4]{GK}. To do so, we introduce some notation. Let $\d=(d_1,d_2,d_3) \in \NN^{3}$, with %\giulio{$1,2\leq d_1$?} 
$1 \leq d_1\leq d_2 \leq d_3$, and let $1 \leq D \leq  d_1+d_2+d_3-3$. Let $a \in \{1,2,3\}$ be such that $\sum_{i=1}^{a-1} (d_i-1) < D \leq \sum_{i=1}^a (d_i-1)$. We define a new sequence $\c=(c_1,c_2,c_3)$ as 
\[
c_i =\begin{cases} 1 & \text{ if } 1 \leq i < a \\
d_a - \left(D-\sum_{i=1}^{a-1} (d_i-1) \right) & \text{ if } i=a \\
d_i & \text{ if } a < i \leq 3.
\end{cases}
\]
For example, if $(\d;D) = (4,4,4;4)$, then $\c=(1,3,4)$.

\begin{corollary} \label{Corollary multiplicity ACI codim 3} Let $\a \subseteq S=\kk[x_1,\ldots, x_n]$ be an almost complete intersection of degrees $(\d;D)=(d_1,d_2,d_3;D)$, with $D \leq \sigma=d_1+d_2+d_3-3$. Then $\e(S/\a) \leq d_1d_2d_3 - c_1c_2c_3$.
\end{corollary}
\begin{proof}
By Theorem \ref{Theorem EGH ACI} we have that $\HF(S/\a) \leq \HF(S/\LL(\d;D))$. Therefore, in order to obtain an upper bound for the multiplicity of $S/\a$, we may replace $\a$ by $\LL(\d;D) = (\x^\d) + (U_D)$. Since $D \leq \sigma$, the variable $x_i$ does not divide $U_D$ for any $i \geq 4$. Thus, after going modulo the regular sequence $x_4,\ldots,x_n$, we may assume that $S/\LL(\d;D)$ is Artinian. With the notation introduced above, one can easily check that $(\x^\d):U_D = (\x^\c)$. It then immediately follows that $\e(S/\LL(\d;D)) = \e(S/(\x^\d)) - \e(S/((\x^\d):U_D)) = d_1d_2d_3 - c_1c_2c_3$.
\end{proof}

Now that we have obtained Corollary \ref{Corollary multiplicity ACI codim 3}, the proof of the Cayley-Bacharach theorem in $\PP^3$, Corollary \ref{Corollary P3}, is immediate. In fact, in the notation of the Corollary, let $\f\subseteq S=\kk[\PP^3]$ be a complete intersection defining $\Gamma$, and $G$ be a form of degree $D$ defining $X$. If $G \notin \f$, by Corollary \ref{Corollary multiplicity ACI codim 3} the multiplicity of $S/(\f+(G))$ is at most $d_1d_2d_3 - c_1c_2c_3$. So $X$ contains at most this number of points of $\Gamma$, which contradicts the assumptions of Corollary \ref{Corollary P3}. 

\begin{comment}
\begin{corollary} \label{Theorem points P3} Let $\Gamma \subseteq \PP^3$ be a complete intersection of forms of degree $\d=(d_1,d_2,d_3)$. If $X$ is a hypersurface of degree $D \leq \sigma=d_1+d_2+d_3-3$ that passes through at least $d_1d_2d_3-c_1c_2c_3+1$ points of $\Gamma$, then $X$ contains $\Gamma$.
\end{corollary}

\ale{Togliere esempio?}

We conclude the section with an explicit example.

\begin{example} Let $\Gamma \subseteq \PP^3$ be a complete intersection of degrees $\d=(d,d,d)$. If $X$ is a surface of degree $d$ passing through at least $d^3-(d-1)d + 1$ points of $\Gamma$, then it contains $\Gamma$. For instance, if a quartic contains at least $53$ points of a complete intersection of three quartics, then it contains all $64$ of them. 
\end{example}
\end{comment}

\section{Almost complete intersections and Cayley-Bacharach theorems in $\PP^n$} \label{Section ACI} 

In order to obtain a Cayley-Bacharach type theorem for points in $\PP^n$, we need to exhibit upper bounds on the multiplicity of almost complete intersections of height $n$. The strategy is to use Theorem \ref{Theorem Artinian} to first reduce to the Artinian case, and then to repeatedly apply a result on the EGH conjecture due to Francisco \cite{F}, together with some symmetry considerations on certain Hilbert functions. This combination of techniques allows us to significantly improve the known upper bounds due to Engheta \cite[Theorem 1]{E}, and later extended by Huneke, Mantero, McCullough and Seceleanu to a more general setting \cite[Theorem 2.2]{HMMCS}.

We start with an easy observation on multiplicities of unmixed ideals. Given a homogeneous ideal $I$, we let $\Assh(S/I) = \{P \in \Ass(S/I) \mid \dim(S/I) = \dim(S/P)\}$. An ideal is called unmixed if $\Assh(S/I) = \Ass(S/I)$.

\begin{remark} \label{Remark multiplicity unmixed}
If $J$ is an unmixed homogeneous ideal of height $h$, and $I$ is a homogeneous ideal of height $h$ which strictly contains $J$, then $\e(S/J) > \e(S/I)$.  In fact, there must exist $P \in \Assh(S/J) = \Ass(S/J)$ such that $J_P \subsetneq I_P$, otherwise the two ideals would coincide. As $J$ is unmixed, and $\Assh(S/I) \subseteq \Assh(S/J) = \Ass(S/J)$, the associativity formula for multiplicities (for instance, see \cite[Theorem 11.2.4]{SH}) gives 
\begin{align*}
\e(S/J) & = \sum_{P \in \Assh(S/J)} \e(S/P) \ell((S/J)_P) \\
& > \sum_{P \in \Assh(S/J)} \e(S/P) \ell((S/I)_P)  \geq \sum_{Q \in \Assh(S/I)} \e(S/Q) \ell((S/I)_Q) = \e(S/I).
\end{align*}
\end{remark}

\begin{notation}
Let $\d=(d_1,\ldots,d_h)$, with $1 \leq d_1 \leq \cdots \leq d_h$. For $m \geq 2$, we consider the $\d$-LPP ideal $\LL(\d;m-1)$ inside $\ov{S}=\kk[x_1,\ldots,x_h]$. Let $\sigma=\sum_{i=1}^h(d_i-1)$, and define 
\[
\varphi_m = \begin{cases} \HF(\ov{S}/(\x^\d);m) - \HF(\ov{S}/\LL(\d;m-1);m) & \text{ if } 2\leq m \leq \sigma\\
0 & \text{ otherwise. }
\end{cases}
\]
\end{notation} 
Clearly, $\varphi_m$ only depends on $m$ and on the given sequence $(\d;D)$. Moreover, observe that for $2 \leq m \leq \sigma$ we have $\varphi_m = \HF(\LL(\d;m-1)/(\x^\d);m)=h-\dim_\kk(((\x^\d) \cap (U_{m-1}))_m)$. In particular, $\varphi_m>0$ for $2 \leq m \leq \sigma$.

\begin{theorem} \label{THM multiplicity}
Let $\a \subseteq S=\kk[x_1,\ldots,x_n]$ be an Artinian almost complete intersection of degrees $(\d;D)=(d_1,\ldots,d_n;D)$, with $D \leq \sigma = \sum_{i=1}^n(d_i-1)$. Then $\HF(S/\a;m) \leq \HF(S/(\x^\d);m)-\varphi_m$ for all $D < m \leq \sigma$. In particular, $\e(S/\a) \leq \prod_{i=1}^n d_i - \sum_{m=D+1}^{\sigma} \varphi_m - 1$. 
\end{theorem}
\begin{proof}

Without loss of generality we may assume that $\kk$ is infinite. We first prove the inequality on Hilbert functions.

We start by treating the case $D<d_1$. Under this assumption, we can find a regular sequence $f_1',\ldots,f_n'$ of degrees $\d'=(D,d_2,\ldots,d_n)$ inside $\a$. To see this, pick $G$ as the first element $f_1'$. Since $\a_{\leq d_2}$ has height at least two, we may find an element $f_2'$ of degree $d_2$ which is regular modulo $f_1'$. Proceding this way, we construct an ideal $\f' \subseteq \a$ generated by a regular sequence of degrees $\d'$. Observe that $\HF(S/\a) \leq \HF(S/\f') = \HF(S/(\x^{\d'}))$. Moreover, since $(\x^{\d'}) = (\x^\d) + (x_1^D)$ is a $\d$-LPP almost complete intersection, we have that $\LL(\d;m-1) \subseteq (\x^{\d'})$ for all $D< m \leq \sigma$. As a consequence, $\HF(S/\a;m) \leq \HF(S/\LL(\d;m-1);m) = \HF(S/(\x^\d);m) - \varphi_m$ for all $D< m \leq \sigma$. 

Now assume that $D \geq d_1$. If $D=\sigma$, there is nothing to show, so we may assume that $D<\sigma$. We proceed by induction on $\sigma' = \sigma-D \geq 1$. Assume that $\sigma'=1$, and observe that $\varphi_\sigma=1$. Since $S/\f$ is an Artinian complete intersection, and $\f \subsetneq \a$, the socle of $S/\f$ must be contained in $\a/\f$. Thus $\HF(S/\a;\sigma) = 0 = \HF(S/\f;\sigma) - 1 = \HF(S/(\x^\d);\sigma)- \varphi_\sigma$, and the desired inequality holds in this case.

Assume $\sigma'>1$. By \cite[Corollary 5.2]{F}, we get $\HF(S/\a;D+1)  \leq \HF(S/\LL(\d;D);D+1)  = \HF(S/(\x^\d);D+1) - \varphi_{D+1}$. We have already observed that $\varphi_{D+1} > 0$, since $D+1 \leq \sigma$. In particular, the above inequality implies that $\HF(\a;D+1) > \HF(\f;D+1)$. Therefore, there exists an element $G' \in \a$ of degree $D+1$ which does not belong to $\f$. Let $\a' = \f + (G')$, which is an almost complete intersection of degrees $(\d;D+1)$. By induction, and because $\a'\subseteq \a$, we have that $\HF(S/\a;m) \leq \HF(S/\a';m) \leq \HF(S/(\x^\d);m) - \varphi_m$ for all $D+1 < m \leq \sigma$, and this concludes the proof of the claimed inequalities on Hilbert function. 

Finally, to obtain the inequality for the multiplicity, it is sufficient to observe that 
\begin{align*}
\e(S/\a) &= \sum_{m=0}^\sigma \HF(S/\a;m) = \sum_{m=0}^{D} \HF(S/\f;m) -1 + \sum_{m=D+1}^\sigma \HF(S/\a) \\
& \leq \sum_{m=0}^\sigma \HF(S/(\x^\d);m)  - \sum_{m=D+1}^\sigma \varphi_m -1 = \prod_{i=1}^n d_i - \sum_{m=D+1}^\sigma \varphi_m -1. \qedhere
\end{align*}
\end{proof}

\begin{remark} \label{Remark HMMCS} 
The bound obtained in Theorem \ref{THM multiplicity}, together with Remark \ref{Remark multiplicity unmixed}, recovers and improves the one given in \cite{HMMCS} and \cite{E}. In fact, by Theorem \ref{Theorem Artinian} we can first of all reduce to the Artinian case. If $D < \sigma$ then $\sum_{m=D+1}^{\sigma} \varphi_m \geq \sigma -D$, and thus $\prod_{i=1}^h d_i - \sum_{m=D+1}^{\sigma} \varphi_m - 1 \leq \prod_{i=1}^h d_i -\sigma + D-1$, which is the bound given in \cite{E,HMMCS}. When $D \geq \sigma$, the results in \cite{E,HMMCS} just give that $\e(S/\a)\leq \prod_{i=1}^h d_i-1$, which is the bound given by Remark \ref{Remark multiplicity unmixed}. Observe that, in the case $D \geq \sigma$, the bound $\e(S/\a) \leq \prod_{i=1}^h d_i-1$ is also the one predicted by the EGH conjecture.
\end{remark}

We now further improve the bound of Theorem \ref{THM multiplicity} by using that, if $\a=\f+(G)$ is an almost complete intersection, then the ideal $\g=\f:\a$ defines a Gorenstein ring, hence it has symmetric Hilbert function.

\begin{theorem} \label{Theorem symmetric}
Let $\a=\f+(G) \subseteq S=\kk[x_1,\ldots,x_n]$ be an almost complete intersection of degrees $(\d;D)=(d_1,\ldots,d_h;D)$, with $D < \sigma = \sum_{i=1}^h(d_i-1)$. Let $\tau_- = \lfloor \frac{\sigma+D-1}{2} \rfloor$ and $\tau^+ = \lceil \frac{\sigma+D-1}{2} \rceil$. Then $\e(S/\a) \leq \prod_{i=1}^h d_i - \sum_{m=D+1}^{\tau_-} \varphi_m - \sum_{m=D+1}^{\tau^+} \varphi_m-2$. 
\end{theorem}
\begin{proof}
We may assume that $\kk$ is infinite and, by Theorem \ref{Theorem Artinian}, that $S/\a$ is Artinian. Let $\g=\f:\a$. Since $\a/\f \cong S/\g(-D)$, and $S/\g$ is Gorenstein, we have that $\HF(\a/\f;D+m) = \HF(\a/\f;\sigma-m)$ for all $m \in \ZZ$. By Theorem \ref{THM multiplicity} we have that $\HF(\a/\f;m) = \HF(S/\f;m) - \HF(S/\a;m) \geq \varphi_m$ for all $m \geq D+1$. Therefore
\begin{align*}
\e(S/\g) & = \sum_{m=D}^{\tau_-} \HF(\a/\f;m) + \sum_{m=D}^{\tau^+} \HF(\a/\f;m) \\
& \geq \sum_{m=D+1}^{\tau_-} \varphi_m  + \sum_{m=D+1}^{\tau^+} \varphi_m + 2.
\end{align*}
Since $\e(S/\a) = \e(S/\f) - \e(S/\g) = \prod_{i=1}^h d_i  - \e(S/\g)$, the proof is complete.
\end{proof}

\begin{remark}
Since the function $m \mapsto \varphi_m$ is non-increasing for $m \geq 2$, Theorem \ref{Theorem symmetric} always provides a bound at least as effective at the one of Theorem \ref{THM multiplicity}.
\end{remark}

We can finally state the main result of this section. 

\begin{theorem} \label{THM Pn}
Let $\Gamma$ be a complete intersection of degrees $\d=(d_1,\ldots,d_n)$ in $\PP^n$, and $X$ be a hypersurface of degree $D$, with $1 \leq D \leq \sigma = \sum_{i=1}^n(d_i-1)$. Set $\delta(\d;D) = \prod_{i=1}^n d_i - \sum_{m=D+1}^{\tau_-} \varphi_m - \sum_{m=D+1}^{\tau^+} \varphi_m - 1$. If $X$ contains at least $\delta(\d;D)$ points of $\Gamma$, then $X$ contains $\Gamma$.
\end{theorem}

We omit the proof since the strategy is the same as in the case of  $\PP^3$, outlined at the end of Section \ref{Section Gorenstein}. By Theorem \ref{Theorem symmetric}, in fact, we may choose $\delta(\d;D)-1$ as an upper bound for the multiplicity of any almost complete intersection of degrees $(\d;D)$. 

\begin{example} Let $\Gamma \subseteq \PP^{2n}$ be a complete intersection of $2n$ cubics. If $D=3$, with the notation of Theorem \ref{Theorem symmetric} we have that $\sigma=4n$, $\tau_- = \tau^+ = 2n+1$, and $\sum_{m=4}^{2n+1} \varphi_m=3n^2-4n+1$. Therefore, if $X$ is a cubic containing at least $3^{2n} - (6n^2-8n+3)$ points of $\Gamma$, then $X$ contains $\Gamma$. For instance, if $\Gamma$ is a complete intersection of four cubics in $\PP^4$, and $X$ is a cubic containing at least $\delta(3,3,3,3;3) = 70$ points of $\Gamma$, then it contains all $81$ points of $\Gamma$. Observe that the optimal value given by the EGH conjecture in this case would be $64$.
\end{example} 

We conclude the section showing that, if either $h = 3$, or $h \geq 4$ and $D < d_4$, then we can improve Theorem \ref{THM Pn} using the results from Section \ref{Section Gorenstein}. In fact, with the notation of Section \ref{Section Gorenstein}, if $h=3$ one can take $\delta(\d;D) = d_1d_2d_3 - c_1c_2c_3+1$, by Corollary \ref{Corollary multiplicity ACI codim 3}. This is a more convenient choice than the value of $\delta(\d;D)$ coming from Theorem \ref{THM Pn}, since it comes from the sharper estimates of Section \ref{Section Gorenstein} on Hilbert functions, which only work for almost complete intersections of height three.  If $h \geq 4$ and $D < d_4$, we have the following theorem.

\begin{theorem} \label{Theorem delta2}
Let $\a =\f+(G)\subseteq S = \kk[x_1,\ldots,x_n]$ be an almost complete intersection of degrees $(\d;D) = (d_1,\ldots,d_h;D)$. Assume that $h \geq 4$ and $D < d_4$. Let $\sigma=\sum_{i=1}^h(d_i-1)$, $\tau_- = \lfloor \frac{\sigma+D-1}{2} \rfloor$ and $\tau^+ = \lceil \frac{\sigma+D-1}{2} \rceil$. Consider the ideal $\LL(\d;D)$ inside $\ov{S} = \kk[x_1,\ldots,x_h]$, and let
\[
\delta_m = \begin{cases} \HF(\ov{S}/(\x^\d);m) - \HF(\ov{S}/\LL(\d;D);m) & \text{ for } 0 \leq m \leq d_4 \\
\varphi_m & \text{ otherwise. }
\end{cases}
\]
Then $\e(S/\a) \leq \prod_{i=1}^h d_i - \sum_{m=D+1}^{\tau_-} \delta_m - \sum_{m=D+1}^{\tau^+} \delta_m-2$. 
\end{theorem}
\begin{proof}
We can assume that $\kk$ is infinite and, by Theorem \ref{Theorem Artinian}, that $S=\ov{S}$ and $h=n$. 

We start by showing that $\HF(S/\a;m) \leq \HF(S/(\x^\d);m) - \delta_m$ for all $m \in \ZZ$. As in the proof of Theorem \ref{Theorem symmetric}, this will yield the desired upper bound for $\e(S/\a)$ since the Hilbert function of $\a/\f$ is symmetric and $\e(S/(\x^\d)) = \prod_{i=1}^n d_i$.

Observe that $\HF(S/\a;m) \leq \HF(S/(\x^\d);m) - \delta_m$ is true for $m>d_4$ by Theorem \ref{THM multiplicity}. Therefore, it suffices to focus on the inequality in degrees $0 \leq m \leq d_4$.

Let $s=\max\{j \geq 4 \mid d_j=d_4\}$. First, assume that the elements $f_1,f_2,f_3,G$ form a regular sequence. Then $\a$ contains a regular sequence $f_1',\ldots,f_n'$ of degrees $\d'=(d_1,d_2,d_3,D,d_5,\ldots,d_n)$. We have $\HF(S/\a) \leq \HF(S/(\x^{\d'})) \leq \HF(S/\LL(\d;D))$, because $\LL(\d;D) \subseteq \LL$, where $\LL$ is the $\d$-LPP ideal with the same Hilbert function as $(\x^{\d'})$, which exists by \cite[Theorem 1.2]{MP}. By definition, $\HF(S/\LL(\d;D);m) = \HF(S/(\x^\d);m)-\delta_m$ for all $0 \leq m \leq d_4$.

If the elements $f_1,f_2,f_3,G$ do not form a regular sequence, then $\b=(f_1,f_2,f_3,G)$ is an almost complete intersection of degrees $(\d'';D) = (d_1,d_2,d_3;D)$. Using Theorem \ref{Theorem EGH ACI} and Remark \ref{Remark over socle} we have that $\HF(S/\b) \leq \HF(S/\LL(\d'';D))$. For all $m < d_4$ we conclude that $\HF(S/\a;m) = \HF(S/\b;m) \leq \HF(S/\LL(\d'';D);m) = \HF(S/\LL(\d;D);m) =\HF(S/(\x^\d);m) - \delta_m$. For $m=d_4$, observe that the elements $f_4,\ldots,f_s$ are all minimal generators of $\a_{\leq d_4}$, so that $\HF(\a/\b;d_4) = s-3$. As a consequence, we get that $\HF(S/\a;d_4) = \HF(S/\b;d_4) - \HF(\a/\b;d_4) \leq \HF(S/\LL(\d'';D);d_4) - (s-3) = \HF(S/\LL(\d;D);d_4)$. 
\end{proof}

\begin{remark} \label{Remark delta2}
Theorem \ref{Theorem delta2} shows that, in the case $h \geq 4$ and $D < d_4$, we may replace the value $\delta(\d;D)$ from Theorem \ref{THM Pn} with $\prod_{i=1}^n d_i - \sum_{m=D+1}^{\tau_-} \delta_m - \sum_{m=D+1}^{\tau^+} \delta_m - 1$. As in the case $h=3$, the latter choice is always more convenient to make, whenever possible. This becomes significantly more evident when $d_4 \gg D$, as the following example shows.
\end{remark}

\begin{example} \label{Example (4,4,4,10;4)} Let $\Gamma \subseteq \PP^4$ be a complete intersection of degrees $\d=(4,4,4,10)$. By Remark \ref{Remark delta2}, if $X$ is a quartic passing through at least $532$ points of $\Gamma$, then $X$ contains all $640$ points of $\Gamma$. Notice that Theorem \ref{THM Pn} would give a value of $\delta(\d;4) = 612$, while the one predicted by the EGH conjecture would be $521$ points. 
\end{example}

\bibliographystyle{alpha}
\bibliography{References}
\end{document}